\documentclass{article}

\usepackage{amsmath, amssymb, amsfonts, ifthen, verbatim}
\usepackage{graphicx}

\newtheorem{theorem}{Theorem}[section]

\newtheorem{lemma}[theorem]{\bf Lemma}

\newtheorem{observation}[theorem]{\bf Observation}

\def\Z{\mathbb{Z}}

\newcommand{\qed}{\hfill$\square$\bigskip}
\newenvironment{proof}[1][]%
{\ifthenelse{\equal{#1}{}}{\noindent\textit{Proof.
    }}{\noindent\textit{#1. }}}%
{\qed}

\title{A note on counting flows in signed graphs}

\author{%
   Matt DeVos\thanks{%
     Email: {\tt mdevos@sfu.ca}. 
     Math Dept., Simon Fraser University, Burnaby, British Columbia, Canada. 
     Supported by an NSERC Discovery Grant (Canada).}
 \and
   Edita Rollov\'{a}\thanks{%
  Email: {\tt rollova@ntis.zcu.cz}. 
  European Centre of Excellence, NTIS – New Technologies for Information Society, Faculty of Applied Sciences, University of West Bohemia, Pilsen. 
  Partially supported by project GA14-19503S of the Czech Science Foundation. 
  Partially supported by project LO1506 of the Czech Ministry of Education, Youth and Sports. 
}
 \and
   Robert \v{S}\'{a}mal\thanks{%
  Email: {\tt samal@iuuk.mff.cuni.cz}. 
  Computer Science Institute of Charles University, Prague. 
  Partially supported by grant GA \v{C}R P202-12-G061.
  Partially supported by grant LL1201 ERC CZ of the Czech Ministry of Education, Youth and Sports.
}
}

\date{}

\let\phi\varphi 
\begin{document}

\maketitle

\begin{abstract}
Tutte initiated  the study of nowhere-zero flows and proved the following fundamental theorem:  For every graph $G$ there is a polynomial $f$ so that
for every abelian group $\Gamma$ of order $n$, the number of nowhere-zero $\Gamma$-flows in $G$ is $f(n)$.  For signed graphs (which have bidirected
orientations), the situation is more subtle.  For a finite group~$\Gamma$, let $\epsilon_2(\Gamma)$ be the largest integer $d$ so that $\Gamma$ has a
subgroup isomorphic to~$\mathbb{Z}_2^d$.  We prove that for every signed graph $G$ and $d \ge 0$ there is a polynomial $f_d$ so that $f_d(n)$ is the
number of nowhere-zero $\Gamma$-flows in $G$ for every abelian group~$\Gamma$ with $\epsilon_2(\Gamma) = d$ and $|\Gamma| = 2^d n$. Beck and
Zaslavsky~\cite{BZ06} had previously established the special case of this result when $d=0$ (i.e., when $\Gamma$ has odd order).  
\end{abstract}


\section{Introduction}
Throughout the paper we permit graphs to have both multiple edges and loops.  Let $G$ be a graph equipped with an orientation of its edges and let $\Gamma$ be an abelian
group written additively.  We say that a function $\phi : E(G) \rightarrow \Gamma$ is a $\Gamma$-\emph{flow} if it satisfies the following equation (Kirchhoff's law) for
every vertex $v \in V(G)$.    
\[ \sum_{e \in \delta^+(v) } \phi(e) - \sum_{e \in \delta^-(v)} \phi(e) = 0, \]
where $\delta^+(v)$ ($\delta^-(v))$ denote the set of edges directed away from (toward) the vertex $v$.
We say that $\phi$ is \emph{nowhere-zero} if $0 \not\in \phi(E(G))$.  
If $\phi$ is a $\Gamma$-flow and we switch the direction of an edge $e$ of $G$, we may obtain a
new flow by replacing~$\phi(e)$ by its additive inverse. Note that this does not affect the property of being nowhere-zero
.  So, in particular, whenever some
orientation of $G$ has a nowhere-zero $\Gamma$-flow, the same will be true for every orientation.  More generally, the number of nowhere-zero $\Gamma$-flows 
in two different orientations of $G$ will always be equal, and we denote this important quantity by $\Phi(G,\Gamma)$. 
Tutte~\cite{Tutte54} introduced the concept of a nowhere-zero $\Gamma$-flow and proved the following key theorem about counting them.

\begin{theorem}[Tutte~\cite{Tutte54}]
\label{tutte}
Let $G$ be a graph.
\begin{enumerate}
\item If $\Gamma $and $\Gamma'$ are abelian groups with $|\Gamma| = |\Gamma'|$, then $\Phi(G,\Gamma) = \Phi(G,\Gamma')$.  
\item There exists a polynomial $f$ so that $\Phi(G,\Gamma) = f(n)$ for every abelian group~$\Gamma$ with $|\Gamma| = n$.
\end{enumerate}
\end{theorem}

Our interest in this paper is in counting nowhere-zero $\Gamma$-flows in signed graphs, so we proceed with an introduction to this setting.  A \emph{signature} of a graph $G$ is a function $\sigma :
E(G) \rightarrow \{-1,1\}$.  We say that a subgraph $H$ is \emph{positive} if $\prod_{e \in E(H)} \sigma(e) = 1$ and \emph{negative} if this product is $-1$, in particular we call
an edge $e$ \emph{positive} (\emph{negative}) if the graph $e$ induces is positive (negative).  We say that two signatures $\sigma$ and $\sigma'$ are \emph{equivalent} if
the symmetric difference of the negative edges of $\sigma$ and the negative edges of $\sigma'$ is an edge-cut of $G$.  Let us note that two signatures are equivalent if
and only if they give rise to the same set of negative cycles; this instructive exercise was observed by Zaslavsky~\cite{Zaslavsky}.  Observe that if $\sigma$ is a signature and $C$ is an
edge-cut of $G$, then we may form a new signature $\sigma'$ equivalent to $\sigma$ by the following rule:
\[ \sigma'(e) = \left\{ \begin{array}{cl}
	\sigma(e)	&	\mbox{if $e \not\in C$}	\\
	- \sigma(e)	&	\mbox{if $e \in C$.}
	\end{array} \right. \]
So, in particular, for any signature $\sigma$ and a non-loop edge $e$, there is a signature $\sigma'$ equivalent to $\sigma$ with $\sigma'(e) = 1$.  We define a 
\emph{signed graph} to consist of a graph~$G$ together with a signature $\sigma_G$. As suggested by our terminology, we will only be interested in properties of signed
graphs which are invariant under changing to an equivalent signature.

Following Bouchet~\cite{Bouchet} we now introduce a notion of a half-edge so as to orient a signed graph.  For
every graph $G$ we let $H(G)$ be a set of \emph{half edges} obtained from the set of edges $E(G)$ as follows. Each edge $e=uv$ contains two distinct half edges $h$ and $h'$ incident with $u$ and $v$, respectively. Note that if $u=v$, $e$ is a loop containing two half-edges both incident with $u$.
For a half-edge $h \in H(G)$, we let $e_h$~denote the edge of~$G$ that contains~$h$.
To orient
a signed graph $G$ we will equip each half edge with an arrow and direct it either toward or away from its incident vertex.  Formally, we define an \emph{orientation} of a
signed graph $G$ to be a function $\tau : H(G) \rightarrow \{-1,1\}$ with the property that for every edge $e$ containing the half edges $h,h'$ we have
\[ \tau(h) \tau(h') = - \sigma_G(e).  \]
We think of a half edge $h$ with $\tau(h) = 1$ ($\tau(h) = -1$) to be directed toward (away from) its endpoint.  Note that in the case when $\sigma_G$ is identically 1, both arrows on every half edge are oriented consistently, and this aligns with the usual notion of orientation of an (ordinary) graph.

\begin{figure}[h]
  \centering
\includegraphics[width=9cm]{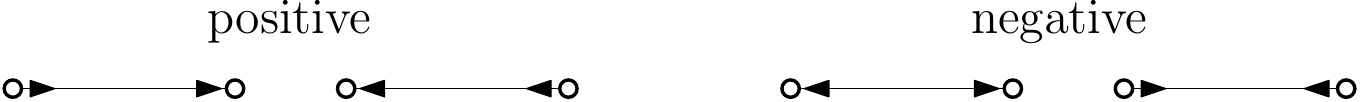}  
\caption{Orientations of edges in a signed graph}
\end{figure}

We define a $\Gamma$-\emph{flow} in such an orientation of a signed graph $G$ to be a function $\phi : E(G) \rightarrow \Gamma$ which obeys the following rule at every vertex $v$
\[  \sum_{ \{h \in H(G) \mid h \sim v \} } \tau(h) \phi(e_h).  \]
As before, we call $\phi$ \emph{nowhere-zero} if $0 \not\in \phi(E(G))$.  Note that in the case when $\sigma_{G}$ is identically 1, this notion agrees with our earlier notion of a (nowhere-zero) flow in an orientation
of a graph.  Also note that, as before, we may obtain a new flow by reversing the orientation of an edge $e$ (i.e., by changing the sign of $\tau(h)$ for both half edges
contained in $e$) and then replacing $\phi(e)$ by its additive inverse.  This new flow is nowhere-zero if and only if the original flow had this
property.  In light of this, we may now define $\Phi(G,\Gamma)$  to be the number of nowhere-zero $\Gamma$-flows in some (and thus every) orientation
of the signed graph~$G$.

As we remarked, we are only interested in properties of signed graphs which are invariant under changing to an equivalent signature, and this is indeed the case for $\Phi(G,\Gamma)$.  To see this, suppose that $\tau$ is an orientation of the signed graph $G$ and that $\phi$ is a nowhere-zero $\Gamma$-flow for this orientation.  Assume that the signature $\sigma'_G$ is obtained from $\sigma_G$ by flipping the sign of every edge in  the edge-cut $\delta(X)$ (here $X \subseteq V(G)$ and $\delta(X)$ is the set of edges with exactly one end in $X$).  Modify the orientation $\tau$ to obtain a new orientation $\tau'$ by switching the sign of $h$ for every half edge incident with a vertex of $X$.  It is straightforward to verify that $\tau'$ is now an orientation of the signed graph given by $G$ and $\sigma_G'$, and $\phi$ is still a $\Gamma$-flow for this new oriented signed graph. 

Beck and Zaslavsky~\cite{BZ06} considered the problem of counting nowhere-zero flows in signed graphs and proved the following analogue of Tutte's Theorem~\ref{tutte} for
groups of odd order.

\begin{theorem}[Beck and Zaslavsky~\cite{BZ06}] 
Let $G$ be a signed graph.
\begin{enumerate}
\item If $\Gamma,\Gamma'$ are abelian groups and $|\Gamma| = |\Gamma'|$ is odd, then $\Phi(G,\Gamma) = \Phi(G,\Gamma')$.  
\item There exists a polynomial $f$ so that for every odd integer $n$, every abelian group $\Gamma$ with $|\Gamma|=n$ satisfies
$f(n) = \Phi(G,\Gamma)$.
\end{enumerate}
\end{theorem}

The purpose of this note is to extend the above theorem to allow for groups of even order by incorporating another parameter.  For any finite group $\Gamma$ we define
$\epsilon_2(\Gamma)$ to be the largest integer $d$ so that $\Gamma$ contains a subgroup isomorphic to $\Z_2^d$ (here $\Z_2 = \Z/2\Z$).

\begin{theorem}
\label{maingroup}
Let $G$ be a signed graph and let $d \ge 0$.  
\begin{enumerate}
\item If $\Gamma$ and $\Gamma'$ are abelian groups with $|\Gamma| = |\Gamma'|$ and $\epsilon_2(\Gamma) = \epsilon_2(\Gamma')$, then $\Phi(G,\Gamma) = \Phi(G,\Gamma')$.  
\item For every nonnegative integer $d$, there exists a polynomial $f_d$ so that $\Phi(G,\Gamma) = f_d(n)$ for every abelian group $\Gamma$ with $\epsilon(\Gamma) = d$
  and $|\Gamma| = 2^dn$.
\end{enumerate}
\end{theorem}

The proof of the above theorem is a straightforward adaptation of Tutte's original method, so it may seem surprising it was not proved earlier.  The cause of this may be
some confusion over whether or not it was already done.  The paper by Beck and Zaslavsky~\cite{BZ06} includes a footnote with the following comment: ``Counting of flows in groups of
even order has been completely resolved by Cameron et al.''.   This refers to an interesting paper of Cameron, Jackson, and Rudd~\cite{CJR} which concerns problems such as counting
the number of orbits of nowhere-zero flows under a group action.  However, the methods developed in this paper only apply to counting nowhere-zero flows in (ordinary) graphs
for the reason that the incidence matrix of an oriented graph is totally unimodular.  Since the corresponding incidence matrices of oriented signed graphs are generally
not totally unimodular (and not equivalent to such matrices under elementary row and column operations), our result does not follow from Cameron et al.  

Before giving the proof of our theorem, let us pause to make one further comment about nowhere-zero flows in signed graphs which consist of a single loop edge~$e$. For a
loop edge $e$ with signature $1$ we may obtain a nowhere-zero flow by assigning any nonzero value $x$ to the edge $e$.  So, two groups $\Gamma$ and $\Gamma'$ will have
the same number of nowhere-zero flows for this graph if and only if $|\Gamma| = |\Gamma'|$.  If, on the other hand, our graph consists of a single loop edge $e$ which is
negative, then the number of nowhere-zero $\Gamma$-flows in this graph will be precisely the number of nonzero group elements $y$ for which $2y = 0$ (i.e., the number of
elements of order~2).  All elements of order~2 form (together with the zero element) a subgroup isomorphic to~$\Z_2^{\epsilon_2(\Gamma)}$, thus 
this number is precisely $2^{{\epsilon}_2(\Gamma)} - 1$.  So, in order for two groups $\Gamma$ and $\Gamma'$ to have the same number of
nowhere-zero flows on this graph, they
must satisfy $\epsilon_2(\Gamma) = \epsilon_2(\Gamma')$. By our main theorem, two groups $\Gamma$ and $\Gamma'$ will satisfy $\Phi(G,\Gamma) = \Phi(G,\Gamma')$ for every signed graph $G$ if and only if this holds for every one edge graph.  This statement is in precise analogy with the situation for flows in ordinary graphs.  

\begin{figure}[h]
  \centering
\includegraphics[width=3cm]{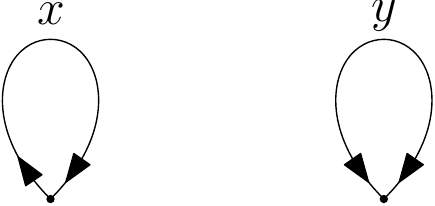}  
\caption{Two graphs that determine $\Phi(G,\Gamma)$ for every other graph~$G$.} 
\end{figure}

We close the introduction by mentioning related results about the number of integer flows. 
Tutte~\cite{Tutte49} defined a nowhere-zero $n$-flow to be a $\Z$-valued flow that only uses values~$k$
with $0 < |k| < n$. Surprisingly, a graph has a nowhere-zero $n$-flow if and only if it has 
a nowhere-zero $\Z_n$-flow. Let us use $\Phi(G,n)$ to denote the number of nowhere-zero $n$-flows on~$G$. 
While $\Phi(G,n)$ and $\Phi(G,\Z_n)$ are either both zero or both nonzero, the actual values differ. 
An analogical statement to the second part of Theorem~\ref{tutte} is again true, by a result of
Kochol~\cite{Kochol}; that is, $\Phi(G,n)$ is a polynomial in~$n$. His result has already been extended for bidirected graphs. 
Beck and Zaslavsky~\cite{BZ06} prove that for a signed graph~$G$, $\Phi (G,n)$~is a quasipolynomial of period 1 or 2; that is, 
there are polynomials~$p_0$ and~$p_1$ such that $\Phi(G,n)$ is equal to~$p_0(n)$ for even~$n$ and to~$p_1(n)$ for odd~$n$. 
Both the Kochol's and the Beck and Zaslavsky's result is proved by an illustrative application of Ehrhart's theorem~\cite{Ehrhart, Sam}. 

\section{The proof}
\label{sec:group}

The proof of our main theorem requires the following lemma about counting certain solutions to an equation in an abelian group.   

\begin{lemma}
\label{abeliancount}
Let $\Gamma$ be an abelian group with $\epsilon_2(\Gamma) = d$ and $|\Gamma| = 2^d n$. Then the number of solutions to $2x_1 + \dots + 2x_t = 0$ with 
$x_1, \ldots, x_t \in \Gamma \setminus \{ 0 \}$ is given by the formula
\[ 
  \sum_{s=0}^{t} (2^d)^s (2^d-1)^{t-s} {t \choose s} \sum_{i=1}^{s-1} (-1)^{i-1} (n-1)^{s-i} \,. 
\]
\end{lemma}

\begin{proof}
We claim that for every abelian group of order $m$, the number of solutions to $x_1 + \dots + x_t = 0$ with $x_1, \ldots, x_t \neq 0$ is given by the formula
\[
  \sum_{i=1}^{t-1} (-1)^{i-1} (m-1)^{t-i} \,. 
\]
We prove this by induction on $t$.  The base case $t=1$ holds trivially.  For the inductive step, we may assume $t \ge 2$.  The total number of solutions to the given
equation for which $x_1, \ldots, x_{t-1}$ are nonzero, but $x_{t}$ is permitted to have any value is exactly $(m-1)^{t-1}$ since we may choose the nonzero terms $x_1,
\ldots, x_{t-1}$ arbitrarily and then set $x_{t} = - \sum_{i=1}^{t-1} x_i$ to obtain a solution.  By induction, there are exactly $\sum_{i=1}^{t-2} (-1)^{i-1}
(m-1)^{t-1-i}$ of these solutions for which $x_{t} = 0$.  We conclude that the number of solutions with all variables nonzero is 
\[ (m-1)^{t-1} - \sum_{i=1}^{t-2} (-1)^{i-1} (m-1)^{t-1-i} = \sum_{i=1}^{t-1} (-1)^{i-1} (m-1)^{t-i} \]
as claimed.  

Now, to prove the lemma, we consider the group homomorphism $\psi : \Gamma \rightarrow \Gamma$ given by the rule $\psi (x) = x+x$.  Note that the kernel of $\psi$,
denoted $ker(\psi)$, is isomorphic to $\Z_2^d$.  Now $x_1, \ldots, x_t$ satisfy $2x_1 + \dots + 2x_t = 0$ if and only if 
$\psi(x_1), \ldots, \psi(x_t)$ satisfy $\psi(x_1) + \dots + \psi(x_t) = 0$.  
So, to count the number of solutions to $2x_1 + \dots + 2x_t = 0$ in $\Gamma$ with all variables nonzero, we may count all possible solutions
to $y_1 + \dots + y_t = 0$ within the group $\psi(\Gamma)$ and then, for each such solution, count the number of nonzero sequences $x_1, \ldots, x_t$ in $\Gamma$ with $
\psi(x_i) = y_i$.  For every $y_i \in \psi(\Gamma)$, the pre-image $\psi^{-1}(y_i)$ is a coset of $ker(\psi)$.  So the number of nonzero elements $x_i$ with $\psi(x_i) =
y_i$ will equal $2^d$ if $y_i \neq 0$ and $2^d - 1$ if $y_i = 0$.  Now we will combine this with the claim proved above.  For every $0 \le s \le t$, the number of
solutions to $y_1 + \dots + y_t=0$ in the group $\psi(\Gamma)$ with exactly $s$ nonzero terms is given by 
\[ 
  {t \choose s} \sum_{i=1}^{s-1} (-1)^{i-1} (n-1)^{s-i} \,.  
\] 
Each such solution will be the image of exactly $(2^d)^s (2^d-1)^{t-s}$ nonzero sequences $x_1, \ldots, x_t \in \Gamma$.  Summing over all $s$ gives the desired formula.  
\end{proof}

We also require the usual contraction-deletion formula for counting nowhere-zero flows.  

\begin{observation}
\label{contdelobs}
Let $G$ be an oriented signed graph and let $e \in E(G)$ satisfy $\sigma_G(e) = 1$.  
\begin{enumerate}
\item If $e$ is a loop edge, then $\Phi(G,\Gamma) = (|\Gamma|-1) \Phi({G \setminus e},\Gamma)$.
\item If $e$ is not a loop edge, then $\Phi(G,\Gamma) = \Phi({G/e},\Gamma) - \Phi({G \setminus e},\Gamma)$.
\end{enumerate}
\end{observation}

\begin{proof} The first part follows from the observation that every nowhere-zero flow in $G$ is obtained from a nowhere-zero flow in $G \setminus e$ by choosing an arbitrary nonzero value for $e$.  The second part follows from the usual contraction-deletion formula for flows.  Suppose $\phi$ is a nowhere-zero flow in $G / e$, and return to the original graph $G$ by uncontracting $e$.  It follows from elementary considerations that there is a unique value $\phi(e)$ we can assign to $e$ so that $\phi$ is a flow.  It follows that $\Phi(G/e,\Gamma)$ is precisely the number of $\Gamma$-flows in $G$ for which all edges except possibly $e$ are nonzero.  This latter count is exactly $\Phi(G,\Gamma) + \Phi(G \setminus e,\Gamma)$ and this completes the proof.
\end{proof}

Equipped with these lemmas, we are ready to prove our main theorem about counting group-valued flows.

\bigskip

\begin{proof}[Proof of Theorem~\ref{maingroup}]
For the first part, we proceed by induction on $|E(G)|$.  Our base cases will consist of one vertex graphs $G$ for which every edge has signature $-1$.  In this case we
may orient $G$ so that every half-edge is directed toward its endpoint.  If the edges are $e_1, \ldots, e_t$, then to find a nowhere-zero flow we need to assign each 
edge~$e_i$ a nonzero value~$x_i$ so that $2x_1 + \dots + 2x_t = 0$.  
By Lemma~\ref{abeliancount}, the number of ways to do this is the same for $\Gamma$ and $\Gamma'$.  

For the inductive step, we may assume $G$ is connected, as otherwise the result follows by applying induction to each component.  If $G$ has a loop edge $e$ with
$\sigma_G(e) = 1$, then the result follows from the previous lemma and induction on $G \setminus e$.  Otherwise $G$ must have a non-loop edge $e$.  By possibly switching
to an equivalent signature, we may assume that $\sigma_G(e) = 1$.  Now our result follows from the previous lemma and induction on $G \setminus e$ and $G / e$.  

The second part of the theorem follows by a very similar argument.  In the base case when $G$ is a one vertex graph in which every edge has signature $-1$, the desired
polynomial is given by Lemma~\ref{abeliancount}.  For the inductive step, we may assume $G$ is connected, as otherwise the result follows by applying induction to each
component and taking the product of these polynomials.  If we are not in the base case, then $G$ must either have a loop edge with signature $1$ or a non-loop edge $e$
which we may assume has signature $1$.  In either case, Observation~\ref{contdelobs} and induction yield the desired result.  
\end{proof}

\bibliographystyle{rs-amsplain}
\bibliography{counting_signed_graph_flows}

\end{document}